\documentclass[a4paper,reqno,10pt]{amsart}

\usepackage[
  colorlinks=true,
  linkcolor=blue,
  citecolor=blue,
  urlcolor=blue
]{hyperref}
\usepackage{a4wide}
\usepackage{amssymb,amsmath,amsthm,amsfonts,mathtools}

\newtheorem{theorem}{Theorem}[section]
\newtheorem{proposition}[theorem]{Proposition}
\newtheorem{lemma}[theorem]{Lemma}
\newtheorem{corollary}[theorem]{Corollary}
\theoremstyle{definition}
\newtheorem{definition}[theorem]{Definition}
\newtheorem{construction}[theorem]{Construction}
\newtheorem{remark}[theorem]{Remark}
\newtheorem*{theoremA}{Theorem A}
\newtheorem*{theoremB}{Theorem B}
\newtheorem*{organization}{Organization}
\newtheorem*{conventions}{Conventions and notation}
\newtheorem*{aiuse}{Use of AI}

\newcommand{\kfield}{k}
\newcommand{\rad}{\operatorname{rad}}
\newcommand{\End}{\operatorname{End}}
\newcommand{\ind}{\operatorname{ind}}
\newcommand{\gldim}{\operatorname{gl.dim}}

\title[The Cartan determinant conjecture]
{The Cartan determinant conjecture for representation-finite algebras}

\author[H.\ Enomoto]{Haruhisa Enomoto}
\address{Parakeet Inc., Japan}
\email{haruhisa.enomoto.math@gmail.com}

\date{\today}

\begin{document}

\begin{abstract}
Let $A$ be a finite-dimensional representation-finite algebra over an
algebraically closed field $k$, and let $M$ be a finite-dimensional $A$-module
such that $\operatorname{End}_A(M)$ has finite global dimension.  We prove that
$\operatorname{End}_A(M)$ has Cartan determinant one if $A$ has finite global
dimension or $\operatorname{char}k\neq2$; in particular, the Cartan determinant
conjecture holds for representation-finite algebras over an algebraically
closed field.
\end{abstract}

\maketitle

\section{Introduction}

Let $\kfield$ be an algebraically closed field, let $A$ be a finite-dimensional
$\kfield$-algebra, and write $C(A)$ for its Cartan matrix.  If $A$ has finite
global dimension then $\det C(A)=\pm1$ \cite{Eilenberg}.  The Cartan
determinant conjecture asks for the sharper equality
\[
 \det C(A)=1.
\]

It remains open in general.  It is known for graded algebras, for
quasi-hereditary and cellular algebras, and for algebras with radical cube
zero, among other classes; see \cite{ChenXi} and the references there.  Wilson
proves that a finite-dimensional algebra of finite global
dimension which carries a positive grading whose positive part is the Jacobson
radical has Cartan determinant one \cite[Corollary~2.3]{Wilson}.  We construct
such a grading, for a representation-finite algebra, from its
Auslander--Reiten quiver.  This gives the following result.

\begin{theoremA}[= Theorem~\ref{thm:main}]
Let $A$ be a finite-dimensional representation-finite algebra over an
algebraically closed field.  If $A$ has finite global dimension, then
\[
 \det C(A)=1.
\]
\end{theoremA}

The same argument proves a more general statement, about endomorphism algebras
of nonzero finite-dimensional modules over representation-finite algebras.
Theorem~A is the case of the regular module.

\begin{theoremB}[= Corollary~\ref{cor:endomorphism}]
Let $A$ be a finite-dimensional representation-finite algebra over an
algebraically closed field $\kfield$, and let $M$ be a nonzero
finite-dimensional right $A$-module with $\gldim\End_A(M)<\infty$.  If
$\gldim A<\infty$, or if $\operatorname{char}\kfield\neq2$, then
\[
 \det C\bigl(\End_A(M)\bigr)=1.
\]
\end{theoremB}

Wilson \cite[Proposition~2.4]{Wilson} and, independently, Igusa--Todorov
\cite[Corollary~7.9]{IgusaTodorov} prove that the Auslander algebra of a
representation-finite algebra has Cartan determinant one.  This does not imply
either of the theorems above: the Cartan matrices there are, up to transpose,
principal submatrices of the Cartan matrix of that Auslander algebra, and the
determinant of a matrix does not determine those of its principal submatrices.

To the best of our knowledge, the Cartan determinant conjecture has not
previously been established for representation-finite algebras over an
algebraically closed field, although it follows by a short argument from the
results of Wilson \cite[Corollary~2.3]{Wilson},
Bautista--Gabriel--Roiter--Salmer\'on \cite{BGRS}, and Lenzing
\cite{Lenzing}, as we shall show in this paper.

\begin{remark}
The proofs of Theorems~A and~B given here depend on the structure theory of
representation-finite algebras developed by
Bautista--Gabriel--Roiter--Salmer\'on \cite{BGRS}.  The author does not know
whether a proof avoiding that theory exists.
\end{remark}

The proof runs as follows.  Both theorems reduce to the case of a standard
algebra, where the indecomposable modules and the maps between them form the
mesh category of the Auslander--Reiten quiver; there $\End_A(M)$ is the
category algebra of a full subcategory on finitely many objects, and path
length grades it positively with the Jacobson radical as positive part, so the
criterion of Wilson applies.  Standardness itself follows from finite global
dimension: a connected representation-finite algebra which is not standard
forces the base field to have characteristic two and has a loop in its ordinary
quiver \cite{BGRS}, and finite global dimension forbids such a loop
\cite{Lenzing}.  In Theorem~B the hypothesis of finite global dimension on $A$
enters only through standardness, which is why a hypothesis on the
characteristic can replace it.

\begin{organization}
Section~\ref{sec:mesh-grading} constructs the path length grading on the
algebra of a full subcategory on finitely many objects of a mesh category and
proves
Theorem~\ref{thm:standard-algebra}.  Section~\ref{sec:standardness} proves that
a representation-finite algebra is standard if it has finite global dimension
or the characteristic is not two.
Section~\ref{sec:consequences} combines these two results to prove Theorems~A
and~B.
\end{organization}

\begin{conventions}
All algebras are finite-dimensional associative unital $\kfield$-algebras, all
modules are finite-dimensional right modules, and $\kfield$ is algebraically
closed.  A module is \emph{basic} if its indecomposable direct summands are
pairwise nonisomorphic, and an algebra is \emph{basic} if it is basic as a
right module over itself.  Cartan determinants, finite global dimension and
representation-finiteness are invariant under Morita equivalence, so we assume
throughout that algebras are basic.  For a basic algebra $A$ with a complete
set $e_1,\ldots,e_n$ of primitive orthogonal idempotents, set
\[
 C(A)=\bigl(\dim_{\kfield}e_iAe_j\bigr)_{i,j=1}^{n}.
\]
Equivalently, $C(A)$ expresses the classes of the indecomposable projective
$A$-modules in the basis of the classes of the simple modules.  Composition of
paths and of morphisms is written right to left, so that $\beta\alpha$ is
defined when the target of $\alpha$ is the source of $\beta$ and is zero
otherwise.  We write $\rad$ for the Jacobson radical and $\gldim$ for global
dimension.  The \emph{ordinary quiver} of an algebra is the quiver of a
presentation of it as a path algebra modulo an admissible ideal.
\end{conventions}

\begin{aiuse}
The proof in this paper was found by GPT-5.6-Sol in a search directed by the
author, and the first drafts of the manuscript were written by GPT-5.6-Sol and
Claude Opus~5.  The author set the structure and much of the wording of the
present text.  The author is responsible for the result.
\end{aiuse}

\section{Positive gradings from mesh categories}
\label{sec:mesh-grading}

We first recall the criterion of Wilson \cite[Corollary~2.3]{Wilson} and then
construct the path length grading needed to apply it to endomorphism algebras
of modules over standard representation-finite algebras.

\begin{theorem}[{Wilson \cite[Corollary~2.3]{Wilson}}]
\label{thm:wilson}
Let $B=\bigoplus_{d\geq 0}B_d$ be a positively graded finite-dimensional
$\kfield$-algebra.  If
\[
 \rad B=\bigoplus_{d>0}B_d
 \qquad\text{and}\qquad
 \gldim B<\infty,
\]
then $\det C(B)=1$.
\end{theorem}

The algebras to which Theorem~\ref{thm:wilson} will be applied are built from
categories with finitely many objects.

\begin{construction}
\label{con:category-algebra}
Let $\mathcal C$ be a $\kfield$-linear category with finitely many objects
$x_1,\ldots,x_s$.  Its \emph{category algebra} is
\[
 B_{\mathcal C}=\bigoplus_{i,j=1}^{s}\mathcal C(x_i,x_j).
\]
For $f:x_i\to x_j$ and $g:x_u\to x_v$, the product $gf$ is
$g\circ f$ when $u=j$ and is zero otherwise.  The identity of
$B_{\mathcal C}$ is $\sum_i 1_{x_i}$.  Put $e_i=1_{x_i}$.  These are orthogonal
idempotents summing to the identity, and
\[
 e_jB_{\mathcal C}e_i=\mathcal C(x_i,x_j).
\]
\end{construction}

Those categories are full subcategories of mesh categories, which we now recall
together with the standard algebras they describe.

\begin{definition}[Mesh categories and standard algebras
{\cite[Section~5.1]{Bongartz}, \cite[Section~3.1]{BretscherGabriel}}]
\label{def:standard}
Let $\Gamma$ be a finite translation quiver with translation $\tau$, together
with a choice, for every arrow $\alpha:y\to x$ ending at a nonprojective vertex
$x$, of an arrow $\sigma(\alpha):\tau x\to y$, such that
$\alpha\mapsto\sigma(\alpha)$ is a bijection from the arrows ending at $x$ onto
the arrows starting at $\tau x$.  The \emph{mesh relation} at $x$ is
\[
 m_x=\sum_{\alpha:y\to x}\alpha\circ\sigma(\alpha).
\]
The \emph{mesh category} $\kfield(\Gamma)$ is the $\kfield$-linear path
category of $\Gamma$ modulo the ideal generated by the relations $m_x$.

For a representation-finite algebra $A$, let $\ind A$ be the full subcategory
of $\operatorname{mod}A$ containing one representative of every isomorphism
class of indecomposable right $A$-modules, and let $\Gamma_A$ be its
Auslander--Reiten quiver: the vertices are the objects of $\ind A$, and the
number of arrows from $X$ to $Y$ is $\dim_{\kfield}\rad(X,Y)/\rad^2(X,Y)$,
where $\rad$ denotes the radical of $\operatorname{mod}A$.  It is a finite
translation quiver whose translation is the Auslander--Reiten translate.  The
algebra $A$ is \emph{standard} if the arrows $\sigma(\alpha)$ for $\Gamma_A$
can be chosen so that there is a $\kfield$-linear equivalence
\[
 \kfield(\Gamma_A)\longrightarrow \ind A
\]
which is the identity on vertices and sends the arrows from $X$ to $Y$ to
morphisms whose classes form a basis of $\rad(X,Y)/\rad^2(X,Y)$.
\end{definition}

The grading required by Theorem~\ref{thm:wilson} comes from path length in a
mesh category.  The next lemma constructs it on the algebra of a full
subcategory on finitely many objects.

\begin{lemma}
\label{lem:full-mesh-corner}
Let $\Gamma$ be a finite translation quiver, and let $\mathcal C$ be a full
subcategory of the mesh category $\kfield(\Gamma)$ with finitely many objects.
If $B_{\mathcal C}$ is finite-dimensional, then $B_{\mathcal C}$ admits a
positive grading whose degree zero part is a product of copies of $\kfield$
and whose positive degree part is its Jacobson radical.
\end{lemma}

\begin{proof}
Every mesh relation is homogeneous of path length two.  Hence, for vertices
$x,y$ of $\mathcal C$, the full morphism space
$\mathcal C(x,y)=\kfield(\Gamma)(x,y)$ is the direct sum of its homogeneous
path length pieces.  This remains true even when a representative path passes
through vertices not belonging to $\mathcal C$.  The length of a composite of
two paths is the sum of their lengths, so the direct sum of these gradings is a
grading of $B_{\mathcal C}$.

The only paths of degree zero are the identity paths.  Thus the degree zero part
is $\prod_{x\in\mathcal C} \kfield e_x$.  Let $I$ be the positive degree
part.  Since $B_{\mathcal C}$ is finite-dimensional, only finitely many
degrees occur, and therefore $I$ is nilpotent.  The quotient by $I$ is the
semisimple degree zero part.  Hence $I=\rad B_{\mathcal C}$.
\end{proof}

Standardness identifies an endomorphism algebra over $A$ with the algebra of a
full subcategory of a mesh category; Lemma~\ref{lem:full-mesh-corner} and
Theorem~\ref{thm:wilson} then give the following.

\begin{theorem}
\label{thm:standard-algebra}
Let $A$ be a finite-dimensional standard representation-finite algebra over an
algebraically closed field, and let $M$ be a nonzero finite-dimensional right
$A$-module.  If
$E=\End_A(M)$ has finite global dimension, then
\[
 \det C(E)=1.
\]
\end{theorem}

\begin{proof}
We may assume that $M$ is basic, since $M$ and the direct sum of one copy of
each of its indecomposable direct summands are each direct summands of finite
direct sums of copies of the other, so their endomorphism algebras are Morita
equivalent.  Write
\[
 M=M_1\oplus\cdots\oplus M_s
\]
with $M_1,\ldots,M_s$ pairwise nonisomorphic indecomposable modules.

Replace the $M_i$ by the chosen isomorphic representatives in $\ind A$, and
let $x_i$ be the corresponding vertices of $\Gamma_A$.  Choose an equivalence
$F:\kfield(\Gamma_A)\to\ind A$ witnessing standardness in
Definition~\ref{def:standard}.  It restricts to an equivalence
from the full subcategory
$\mathcal C$ of $\kfield(\Gamma_A)$ on $x_1,\ldots,x_s$ to the full
subcategory of $\ind A$ on $M_1,\ldots,M_s$.  Under this equivalence, the map
\[
 B_{\mathcal C}\longrightarrow \End_A\!\left(\bigoplus_{i=1}^{s}M_i\right),
 \qquad (f_{ji})_{i,j}\longmapsto (F(f_{ji}))_{i,j},
\]
is an algebra isomorphism.  Both sides are written as matrices, with $(j,i)$
entry a morphism $x_i\to x_j$ in $\mathcal C$ on the left and a morphism
$M_i\to M_j$ on the right, and in both the product is formed by the matrix
rule, with composition of morphisms in place of multiplication of entries.  In
particular, $B_{\mathcal C}$ is finite-dimensional.

Lemma~\ref{lem:full-mesh-corner} therefore gives a positive grading
\[
 E=\bigoplus_{d\geq0}E_d,
 \qquad \rad E=\bigoplus_{d>0}E_d.
\]
Theorem~\ref{thm:wilson} now gives $\det C(E)=1$.
\end{proof}

\section{Standardness from finite global dimension}
\label{sec:standardness}

To apply Theorem~\ref{thm:standard-algebra} to a representation-finite algebra
of finite global dimension, it remains to prove that the algebra is standard.
The next proposition isolates the one consequence of \cite{BGRS} that we
use.

\begin{proposition}
\label{prop:standard-form-dichotomy}
Let $A$ be a finite-dimensional connected representation-finite algebra over an
algebraically closed field.  If $A$ is nonstandard, then the characteristic of
the field is two and the ordinary quiver of $A$ has a loop.
\end{proposition}

\begin{proof}
Call a finite-dimensional algebra \emph{distributive} if its lattice of
two-sided ideals is distributive.  Ringel proves that an algebra which is not
distributive has an indecomposable module of every positive dimension
\cite{Ringel}, stated in this form in \cite[Section~2.2]{Bongartz}.  A
representation-finite algebra has only
finitely many indecomposable modules up to isomorphism, so $A$ is distributive.

Bautista--Gabriel--Roiter--Salmer\'on prove that a distributive
representation-finite algebra over an algebraically closed field of
characteristic different from two is standard \cite[p.~218]{BGRS}; see also
\cite[Section~3.1]{Bongartz}.  Suppose therefore that
$\operatorname{char}\kfield=2$ and that $A$ is nonstandard.  The same authors
describe the nonstandard algebras which then occur
\cite[Theorem~9.6]{BGRS}; see also \cite[Section~3.7]{Bongartz}.  Each contains
at least one penny-farthing, a configuration supported on vertices of the
ordinary quiver of $A$ one of which carries a loop; the full shape is recorded
in \cite[Figure~3]{Bongartz}.  The ordinary quiver of $A$ therefore has a
loop.
\end{proof}

Finite global dimension excludes such a loop, by the following result,
historically known as the no-loop conjecture.

\begin{theorem}[{Lenzing \cite{Lenzing}; reproved in \cite{IgusaNoLoops};
stated as in \cite[Theorem~2.2]{HappelZacharia}}]
\label{thm:no-loop}
Let $A$ be a finite-dimensional basic algebra over an algebraically closed
field.  If the ordinary quiver of $A$ contains a loop, then
\[
 \gldim A=\infty.
\]
\end{theorem}

Proposition~\ref{prop:standard-form-dichotomy} concerns connected algebras, so
we reduce to that case before applying it.

\begin{corollary}
\label{cor:finite-gldim-standard}
Let $A$ be a finite-dimensional representation-finite algebra over an
algebraically closed field $\kfield$.  If $\gldim A<\infty$, or if
$\operatorname{char}\kfield\neq2$, then $A$ is standard.
\end{corollary}

\begin{proof}
Decompose $A$ as $A=A_1\times\cdots\times A_r$ with the $A_i$ connected; each
$A_i$ is again representation-finite.  We show that every $A_i$ is standard.
If $\operatorname{char}\kfield\neq2$, this is immediate from
Proposition~\ref{prop:standard-form-dichotomy}, whose conclusion includes
$\operatorname{char}\kfield=2$.  If instead $\gldim A<\infty$, then
$\gldim A_i<\infty$, so Theorem~\ref{thm:no-loop} shows that the ordinary
quiver of $A_i$ has no loop, and the contrapositive of
Proposition~\ref{prop:standard-form-dichotomy} again gives standardness.

The Auslander--Reiten quiver of the product is
the disjoint union of the $\Gamma_{A_i}$, while $\ind A$ and the mesh category
$\kfield(\Gamma_A)$ are the corresponding disjoint unions of categories.  Thus
the equivalences $\kfield(\Gamma_{A_i})\to\ind A_i$ witnessing standardness of
the factors assemble to an equivalence $\kfield(\Gamma_A)\to\ind A$ witnessing
standardness of $A$.
\end{proof}

\section{Proofs of the main theorems}
\label{sec:consequences}

We now combine Corollary~\ref{cor:finite-gldim-standard} with
Theorem~\ref{thm:standard-algebra} to prove the two results announced in the
Introduction.

\begin{theorem}
\label{thm:main}
Let $A$ be a finite-dimensional representation-finite algebra over an
algebraically closed field.  If $\gldim A<\infty$, then
\[
 \det C(A)=1.
\]
\end{theorem}

\begin{proof}
By Corollary~\ref{cor:finite-gldim-standard}, $A$ is standard.  Apply
Theorem~\ref{thm:standard-algebra} to $A$ with the regular module $M=A_A$.
Since $\End_A(A_A)\cong A$, the conclusion is $\det C(A)=1$.
\end{proof}

The following statement is more general than Theorem~\ref{thm:main}.

\begin{corollary}
\label{cor:endomorphism}
Let $A$ be a finite-dimensional representation-finite algebra over an
algebraically closed field $\kfield$, and let $M$ be a nonzero
finite-dimensional right $A$-module with $\gldim\End_A(M)<\infty$.  If
$\gldim A<\infty$, or if $\operatorname{char}\kfield\neq2$, then
\[
 \det C\bigl(\End_A(M)\bigr)=1.
\]
\end{corollary}

\begin{proof}
By Corollary~\ref{cor:finite-gldim-standard}, $A$ is standard.
Theorem~\ref{thm:standard-algebra} therefore applies to $A$ and $M$.
\end{proof}

\end{document}